\documentclass[11pt,a4paper]{amsart}

\usepackage{amsmath}
\usepackage{amsfonts}
\usepackage{graphicx}
\usepackage{framed}
\usepackage{amssymb,cancel}
\usepackage{xcolor}
\usepackage{esint}
\usepackage{comment}

\usepackage{etex}
\usepackage{latexsym}
\usepackage[english]{babel}
\usepackage[utf8x]{inputenc}

\def \R {\mathbb{R}}

\def \e {\varepsilon}

\usepackage{amsthm}
\theoremstyle{definition}
\newtheorem{definition}{Definition}[section]

\newtheorem{remark}[definition]{Remark}

\theoremstyle{plain}
\newtheorem{theorem}[definition]{Theorem}
\newtheorem{proposition}[definition]{Proposition}
\newtheorem{lemma}[definition]{Lemma}

\numberwithin{equation}{section}
\usepackage[colorlinks=true,urlcolor=blue,
citecolor=red,linkcolor=blue,linktocpage,pdfpagelabels,
bookmarksnumbered,bookmarksopen]{hyperref}
\usepackage[hyperpageref]{backref}

\makeatletter
\@namedef{subjclassname@2020}{\textup{2020} Mathematics Subject Classification}
\makeatother

\begin{document}

\title[Local and nonlocal critical problem]{On a Sobolev critical problem \\ for the superposition \\ of a local and nonlocal operator \\  with the ``wrong sign''}

  \author[S.\,Biagi]{Stefano Biagi}
 \author[S.\,Dipierro]{Serena Dipierro}
 \author[E.\,Valdinoci]{Enrico Valdinoci}
 \author[E.\,Vecchi]{Eugenio Vecchi}
 
 \address[S.\,Biagi]{Dipartimento di Matematica
 	\newline\indent Politecnico di Milano \newline\indent
 	Via Bonardi 9, 20133 Milano, Italy}
 \email{stefano.biagi@polimi.it}
 
 \address[S.\,Dipierro]{Department of Mathematics and Statistics
 	\newline\indent University of Western Australia \newline\indent
 	35 Stirling Highway, WA 6009 Crawley, Australia}
 \email{serena.dipierro@uwa.edu.au}
 
 \address[E.\,Valdinoci]{Department of Mathematics and Statistics
 	\newline\indent University of Western Australia \newline\indent
 	35 Stirling Highway, WA 6009 Crawley, Australia}
 \email{enrico.valdinoci@uwa.edu.au}
 
 \address[E.\,Vecchi]{Dipartimento di Matematica
 	\newline\indent Università di Bologna \newline\indent
 	Piazza di Porta San Donato 5, 40126 Bologna, Italy}
 \email{eugenio.vecchi2@unibo.it}



\thanks{S. Biagi and E. Vecchi are members of INdAM.
	S. Dipierro and E. Valdinoci are members of AustMS.
	S. Dipierro is supported by
	the Australian Research Council 
	Future Fellowship FT230100333.
	E. Valdinoci is supported by the Australian Laureate Fellowship
	FL190100081. S. Biagi and E. Vecchi are supported partially by the PRIN 2022 project 2022R537CS \emph{$NO^3$ - Nodal Optimization, NOnlinear elliptic equations, NOnlocal geometric problems, with a focus on regularity}, funded by the European Union - Next Generation EU and partially by the Indam-GNAMPA project CUP E5324001950001 - {\em Problemi singolari e degeneri: esistenza, unicità e analisi delle proprietà qualitative delle soluzioni}}

\date{\today}

\begin{abstract} We study a critical problem for an operator of mixed order obtained by the superposition of a Laplacian with a fractional Laplacian.

The main novelty is that we consider a mixed operator of the form~$
-\Delta- \gamma(-\Delta)^s$, namely we suppose that the fractional Laplacian has the ``wrong sign'' and can be seen as a nonlocal perturbation of the purely local case, which is needed to produce a nontrivial solution of the critical problem.
\end{abstract}

\maketitle

\section{Introduction} 
Let $\Omega \subset \mathbb{R}^n$ (with $n \geq 3$) be an open and bounded set with smooth enough (say Lipschitz) boundary $\partial \Omega$.
We consider the following problem
\begin{align}\label{eq:Problem}
	\left\{\begin{array}{rl}
		-\Delta u = \gamma \, (-\Delta)^s u + |u|^{2^{\ast}-2}u & \textrm{ in } \Omega,\\
		u\not\equiv0 & \textrm{ in } \Omega,\\
		u= 0 & \textrm{ in } \mathbb{R}^{n} \setminus \Omega.
	\end{array}\right.
\end{align}
The operator~$(-\Delta)^s$ is the fractional Laplacian, defined, for~$s\in(0,1)$, as
$$(-\Delta)^s u(x):=\int_{\R^n}
\frac{u(x)-u(x+y)-u(x-y)}{|y|^{n+2s}}\,dy.$$
The parameter $\gamma$ is assumed to be positive
and suitably small (namely, $\gamma \in (0, C_{emb})$, with a notation that will be clarified here below). 
As usual, $2^{\ast} := \tfrac{2n}{n-2}$ is the Sobolev critical exponent. Our aim is to prove that under the previous assumption on $\gamma$, problem \eqref{eq:Problem} admits a nontrivial weak solution in the appropriate variational space. The cases $n=1$ and $n=2$ are usually considered separately. Indeed, the case $n=2$ leads to critical nonlinearities of exponential type closely related to the Moser-Trudinger inequality.
  
  \vspace{0.1cm}
  
  \noindent -\,\,\emph{Assumptions and motivations}.
  It is known, see e.g.~\cite{DPV}, that the Sobolev space $H^{1}(\mathbb{R}^n)$ continuously embeds into the fractional Sobolev space~$H^{s}(\mathbb{R}^n)$ for~$s\in (0,1)$, i.e. there exists a positive constant $C>0$ (depending on $s$ and $n$) such that
  \begin{equation}\label{eq:embedding}
  C \, [u]^{2}_{s}\leq  \|u\|^2_{H^{1}(\mathbb{R}^n)}, 
  \end{equation}
for every $u \in H^{1}(\mathbb{R}^n)$. 
Here and in what follows, we use the notation  
$$ [u]_{s} := \left(\iint_{\mathbb{R}^{2n}}\dfrac{|u(x)-u(y)|^2}{|x-y|^{n+2s}}\, dx \,dy\right)^{1/2}$$
for the Gagliardo seminorm of~$u$. 
  
Together with the classical Poincar\'e inequality, formula~\eqref{eq:embedding} shows that there exists~$C > 0$ (depending on $\Omega$ as well) such that
  \begin{equation} \label{eq:EmbeddingH01Hs} C[{u}]^2_s\leq \int_\Omega|\nabla u|^2\,dx = \|\nabla u\|^2_{L^2(\Omega)},	
  \end{equation} for every $u\in H_0^1(\Omega)$. Hence, we can define
   	\begin{equation*}
  		C_{emb} := \inf \big\{ \|\nabla u\|^{2}_{L^2(\Omega)} : u\in H_0^1(\Omega),\,[u]_{s}^2 =1\big\}.
  \end{equation*}
  
  The assumption $\gamma \in (0, C_{emb})$ can now be explained: we are interested in dealing with a local and nonlocal operator {\it with the wrong sign}, namely
  \begin{equation}\label{eq:Operator_wrong_sign}
  		-\Delta  - \gamma \, (-\Delta)^s,
  \end{equation}
 but at the same time we need some restrictions on the coefficient $\gamma$ so that the operator is positive definite.
 As a ``historical'' comment, we are unsure who first introduced the terminology ``wrong sign'' to denote superpositions of operators with coefficients of different signs; however, this jargon already appears in~\cite{BDVV-CPDE22}.
 	The jargon ``mixed local and nonlocal operator'' was also formerly used in~\cite{BDVV2-2021, DV-niche}.
 	Mixed operators with the ``wrong sign'' were also studied specifically in~\cite{PESP:2024} and referred to as
 	``local and
 	nonlocal operators'' (however, \cite{PESP:2024}
 	cannot include the specific case of a Laplacian minus the fractional Laplacian, see the assumptions
 	after~(1.1) there). We also mention \cite{MMV, GMV, CCMV} where the authors refer to similar operators as {\em indefinite} or {\em nonpositive} operators.
 	
 	An extensive study of
 	L\'evy measures associated with integro-differential operators
 	was presented in~\cite{GARR}. In this spirit, one can also consider superposition
 	operators averaged over a measure.
 	The case of a signed measure can be seen as a generalization to a continuum (or to an infinite series)  of
 	the sum of two operators in which the lower-order one possesses the wrong sign:
 	for this setting, see~\cite{RE2, RE22, RE5, RE6, RE7, RE8, PESP:2024, RE10, RE11, RE12, RE13}.
 	
 	In any case, the study of operators of mixed integer and fractional order is classical
 	(see e.g.~\cite{COURR} in relation with the maximum principle), but it has been recently
 	experiencing an intense revival (see e.g., among the many,
 	\cite[equation~(1.2)]{JA06},
 	\cite[Section~3.2]{SILVE06},
 	\cite[Remark 5.6]{LLAVE09},
 	and~\cite[equation~(2)]{ALI20}).

The main motivation behind the study of problem \eqref{eq:Problem} comes from the recent works~\cite{BDVV5,BaDPQ}. The starting point is the following simple observation: let $\Omega \subseteq \mathbb{R}^{n}$ be an open set and $s \in (0,1)$, then there exists a constant $\mathcal{S}_{n,s}(\Omega)>0$ such that
  \begin{equation}\label{eq:mixed_Sobolev}
  	\mathcal{S}_{n,s}(\Omega)  \|u\|^{2}_{L^{2^{\ast}}(\mathbb{R}^{n})} \leq \|\nabla u\|^{2}_{L^{2}(\mathbb{R}^{n})} + [u]^2_{s},
  \end{equation}
  for every $u \in C^{\infty}_{0}(\mathbb{R}^{n})$. This is clearly a consequence of the classical Sobolev inequality.
  
In principle, the constant $\mathcal{S}_{n,s}(\Omega)$ could depend on~$n$, $s$ and~$\Omega$, but this is actually not the case (see~\cite[Theorem~1.1]{BDVV5}), indeed 
  \begin{equation*}
  	\mathcal{S}_{n,s}(\Omega) = S_{n}
  \end{equation*}
   where $S_n$ is the best constant for the classical Sobolev inequality, achieved by the Aubin-Talenti functions.
  
  Moreover, the energy functional 
  \begin{equation}\label{eq:Mixed_Energy}
  	E(u) := \dfrac{1}{2} \int_{\Omega}|\nabla u|^2 \, dx + \dfrac{1}{2}[u]^{2}_s
  \end{equation}
exhibits a lack of scaling invariance, namely
  \begin{equation}\label{eq:Lack_Scaling}
  	E(u_{t}) = \dfrac{1}{2} \int_{\Omega}|\nabla u|^2 \, dx + \dfrac{t^{2s-2}}{2}[u]^{2}_s, 
  \end{equation}
where $u_{t}(x) := t^{(n-2)/2}u(tx)$ is the rescaled function which preserves the $H^{1}_{0}$-norm of $u$. This second simple remark allows to show that the best constant in the local-nonlocal Sobolev inequality \eqref{eq:mixed_Sobolev} is never achieved, see~\cite[Theorem~1.2]{BDVV5}.

Since the first variation of the energy functional $E$ in \eqref{eq:Mixed_Energy} is given by the local-nonlocal operator 
  \begin{equation*}
  	-\Delta + (-\Delta)^s,
  \end{equation*}  
  it becomes natural to attack the study of Sobolev critical problems in bounded domains, with leading operator $-\Delta + (-\Delta)^s$. The first results in this direction are contained in~\cite{BDVV5,BiagiVecchi,BiagiVecchi2} and are all inspired by the seminal paper~\cite{BrNir}, where the authors showed the importance of adding some kind of perturbation (linear, superlinear, sublinear) when looking for positive weak solutions. In all the mentioned papers~\cite{BDVV5,BiagiVecchi,BiagiVecchi2} the main difficulties (and differences from the purely local case) come once again from the lack of scaling invariance in \eqref{eq:Lack_Scaling}, especially when one has to compute the fractional seminorm of a suitably localized version of the Aubin-Talenti functions, see e.g. (iv) in Lemma~\ref{lem:andamenti} below.
 
 \medskip
 
A natural question (raised to us by S. Terracini) is what happens when one considers a Sobolev critical problem with a leading operator of the form \eqref{eq:Operator_wrong_sign}. This could also be interpreted in the spirit of~\cite{BrNir} as a nonlocal perturbation of the purely local problem.
Our main result
 reads as follows:
 
 \begin{theorem}\label{thm:Main}
 	Let $n\geq 3$ and $\Omega \subset \mathbb{R}^{n}$ be an open and bounded set with Lipschitz boundary.
 	Then, the following assertions hold:
 	\begin{enumerate}
 	\item[(1)] \emph{(}\textbf{High-dimensional case}\emph{)} If $n\geq 5$,	 for every $\gamma\in (0,C_{emb})$ there exists a nontrivial
 	weak solution $u \in \mathcal{X}^{1,2}(\Omega)$ to problem \eqref{eq:Problem}.
 	\vspace{0.1cm}
 	
 	\item[(2)] \emph{(}\textbf{Low-dimensional case}\emph{)} If $n = 3,4$, there exists $\gamma^*\in [0,C_{emb})$ such that
 	problem \eqref{eq:Problem} has a nontrivial
 	weak solution $u\in\mathcal{X}^{1,2}(\Omega)$ for every~$\gamma\in(\gamma^*,C_{emb})$.
 	\end{enumerate}
 \end{theorem}
 
 We mention~\cite[Corollary 5.4]{RE2}
 	for a result close in spirit to Theorem~\ref{thm:Main}
 	here. The setting in~\cite{RE2} is however structurally different from the one presented here
 	and relies on the analysis of the Dancer-Fu\v{c}\'\i{}k spectrum. Also,
 	the coefficients~$a$ and~$b$ in~\cite{RE2} were considered as strictly positive,
 	while the setting here would correspond to~$a=b=0$, which was not explicitly addressed
 	in~\cite{RE2}.
 	
 We postpone to Section~\ref{sec:Prel}, and in particular to \eqref{eq:def_X12}, the exhaustive discussion on the function space $\mathcal{X}^{1,2}(\Omega)$ and several equivalent definitions. For the moment, we just point out that
 one of the equivalent definitions reads
 $$\mathcal{X}^{1,2}(\Omega) = \big\{u\in H^1(\R^n):\,\text{$u|_\Omega\in H_0^1(\Omega)$ and 
  		$u\equiv 0$ a.e.\,in $\R^n\setminus\Omega$}\big\}.$$
 
The proof of Theorem~\ref{thm:Main} is contained in Section~\ref{sec:mainproof87654}.
  
  \section{Preliminary results}\label{sec:Prel}
  \noindent {\bf Notations.} Throughout the paper, we tacitly
  exploit all the notation listed below; we thus refer the Reader to this list
  for any non-standard notation encountered.
  \begin{itemize}
  	\item $\Omega \subset \mathbb{R}^{n}$ is an open and bounded set with Lipschitz boundary $\partial \Omega$.
  	\item Given any set $E$, $|E|$ denotes its $n$-dimensional Lebesgue measure.
  	\item Given any set $E$, $E^c$ denotes its complement (in $\mathbb{R}^n$), namely~$E^c:=\R^n\setminus E$.
  	\item Given $x_0 \in \mathbb{R}^n$ and $r>0$, we denote by $B_{r}(x_0)$ the Euclidean open ball of center $x_0$ and radius $r$. When $x_0=0$, we will omit it and simply write~$B_r$.
  	\item Given any two vectors $v,w \in \mathbb{R}^{n}$, we denote by
  	$v\cdot w$ the usual scalar product.
  	\item Given any $p \in [1,n)$ we denote by $p^{*}$ the associated Sobolev exponent (with respect to Euclidean space $\mathbb{R}^n$), that is
  	$$p^{*}  := \frac{np}{n-p}.$$
  	\item Given any $r \in (1,+\infty)$, we denote by $r'$ the conjugate exponent of $r$ in the usual H\"older inequality, that is, 
  	$$r'  := \frac{r}{r-1}.$$
  	\item Given a reflexive Banach space $V$ with dual space $V'$, we denote by 
  	$\langle \cdot,\cdot\rangle$
  	the duality product. 
  	\item Given a measurable set $E$, we denote by $\chi_{E}$ its characteristic function.
  	\item Given any function $v$, we denote the positive and negative part of $v$ by 
  	\begin{equation}\label{eq:def_positive_negative_part}
  		v^{+}:= \max\{v,0\} \quad \textrm{ and } \quad v^{-}:= -\min\{v,0\}.
  	\end{equation}
  	\item $S_n$ will denote the best constant in the classical Sobolev inequality which is achieved by the family of Aubin-Talenti functions defined, for every~$y\in\R^n$ and~$\varepsilon>0$, as \begin{equation}\label{eq:def_Talentiane}
  		V_{\varepsilon,y}(x) := \dfrac{\varepsilon^{(n-2)/2}}{(\varepsilon^{2} + |x-y|^2)^{(n-2)/2}}.
  	\end{equation}
  	We will avoid writing the dependence on $y$ when $y=0$.
  \end{itemize}
  
  \medskip 
  
  We now set the adequate functional setting to study problem \eqref{eq:Problem}. We refer to e.g.~\cite[Section~2]{BDVV5} for more details.
  
  Let~$s\in (0,1)$ and let $u:\R^n\to\R$ be a measurable function. We set
  \begin{equation} \label{eq:GagliardoSeminorm}
  	[u]_s := \left(\iint_{\R^{2n}}\frac{|u(x)-u(y)|^2}{|x-y|^{n+2s}}\,dx\,dy\right)^{1/2},
  \end{equation}
  and we refer to $[u]_s$ as the \emph{Gagliardo seminorm} of $u$ (of order $s$).
  
  Let $\varnothing\neq \Omega\subseteq\R^n$ (here $n\geq 3$ would suffice) be an open set, not necessarily bounded. We define the function space
  $\mathcal{X}^{1,2}(\Omega)$ as 
  \begin{equation*}
  	\mathcal{X}^{1,2}(\Omega) := \overline{C_0^\infty(\Omega)}^{\,\,\rho(\cdot)}
  \end{equation*}
   i.e. as  the com\-ple\-tion
  of $C_0^\infty(\Omega)$ with respect to the norm  
  \begin{equation}\label{eq:def_rho}
  	\rho(u) := \left(\|\nabla u\|^2_{L^2(\R^n)}+[u]^2_s\right)^{1/2}.
  \end{equation}
    We stress that the norm $\rho(\cdot)$ is induced by the scalar product
  		$$\langle u,v\rangle_{\rho} := \int_{\R^n}\nabla u\cdot\nabla v\,dx
  		+ \iint_{\R^{2n}}\frac{(u(x)-u(y))(v(x)-v(y))}{|x-y|^{n+2s}}\,dx\,dy,$$
  		and therefore $\mathcal{X}^{1,2}(\Omega)$ is a Hilbert space.

For our purposes, we are interested in the case~$\Omega$ bounded. Recalling the continuous embedding given by~\eqref{eq:embedding}, and combining it with the classical Poincar\'e inequality, we find that
  $\rho(\cdot)$ and the full $H^1$-norm in $\R^n$
  are actually equivalent on the space $C^\infty_0(\Omega)$, and hence
  \begin{equation}\label{eq:def_X12}
  	\begin{aligned}
  	\mathcal{X}^{1,2}(\Omega) & = \overline{C_0^\infty(\Omega)}^{\,\,\|\cdot\|_{H^1(\R^n)}} \\
  	& = \big\{u\in H^1(\R^n):\,\text{$u|_\Omega\in H_0^1(\Omega)$ and 
  		$u\equiv 0$ a.e.\,in $\R^n\setminus\Omega$}\big\}.
  \end{aligned}
\end{equation}

We stress that in \cite[equation~(2.10)]{RE7}, the authors chose a different norm modeled on the energy associated to the local-nonlocal operator. In  our notation their choice reads as 
	$$\left(\|\nabla u\|^2_{L^{2}(\mathbb{R}^n)} - \gamma [u]_s^2\right)^{1/2}.$$
	It turns out that their choice is actually equivalent to the one taken in the present paper, thanks to~\cite[Lemma~2.1]{RE2}. 

\begin{remark} \label{rem:PropX12Usare} On account of \eqref{eq:def_X12}, 
we derive the following useful facts.
\begin{itemize}
\item[i)] $\mathcal{X}^{1,2}(\Omega)$ is \emph{continuously embedded into $L^p(\Omega)$} for every $1\leq p\leq 2^*$; moreover, if $1\leq p < 2^*$, this embedding is \emph{compact}.
\vspace{0.1cm}

\item[ii)] Taking into account the continuous embedding in \eqref{eq:embedding},
and since this embedding \emph{is actually compact} (see, e.g.,~\cite[Theorem~1.3]{emBE}, applied here with~$p:=\widetilde p:=2$, $s:=1$ and~$\widetilde s:=s$),
we get that, for all~$s\in(0,1)$,
$$\text{$\mathcal{X}^{1,2}(\Omega)$ is compactly embedded into $H^s(\mathbb{R}^n)$}.$$
\item[iii)]
Moreover, it is easily seen that
\begin{equation} \label{eq:def_Cemb}
	C_{emb} = \inf\big\{\|\nabla u\|^2_{L^2(\Omega)}:\,u\in\mathcal{X}^{1,2}(\Omega),\,[u]^2_s = 1\big\},
\end{equation}
and this infimum is achieved, that is, there exists a function $\Phi_0\in\mathcal{X}^{1,2}(\Omega)$ such that 
$$[\Phi_0]^2_s = 1\quad\text{and}\quad \|\nabla \Phi_0\|^2_{L^2(\Omega)} = C_{emb}.$$
\end{itemize}
\end{remark}

We end this section by giving the precise definition of
 \emph{weak solution} to~\eqref{eq:Problem}.
  
 \begin{definition}\label{def:WeakSol}
 	We say that a function $u:\mathbb{R}^n\to\mathbb{R}$ is a \emph{weak solution}
 	to problem~\eqref{eq:Problem} if it satisfies the following properties:
 	\begin{itemize}
 		\item[1)] $u\in\mathcal{X}^{1,2}(\Omega)$;
 		\item[2)] for every test function $v\in \mathcal{X}^{1,2}(\Omega)$, we have
 		\begin{align*}
 		 \int_{\R^n}\nabla u\cdot\nabla v\,dx
 		 -\gamma \, \iint_{\R^{2n}}\frac{(u(x)-u(y))(v(x)-v(y))}{|x-y|^{n+2s}}\,dx\,dy = \int_\Omega |u|^{2^*-2}uv\,dx.	
 		\end{align*}
 	\end{itemize}
 \end{definition} 
 
  \section{Proof of Theorem~\ref{thm:Main}}\label{sec:mainproof87654}
  We follow the ideas presented in~\cite{BrNir} combined with the crucial expansion provided by Proposition~\ref{prop:Stima_Migliore}.
  To this aim, let us first define the functional
  \begin{equation}\label{eq:Def_Qgamma}
 	Q_{\gamma}(u) := \int_{\Omega}|\nabla u|^{2} \, dx - \gamma \, [u]^{2}_{s}, \quad u \in \mathcal{X}^{1,2}(\Omega),
  \end{equation}
 and let us consider the following constrained minimization problem
\begin{equation}\label{eq:Def_Sgamma}
		S(\gamma) := \inf \big\{ Q_{\gamma}(u) : u \in \mathcal{X}^{1,2}(\Omega), \|u\|_{L^{2^{\ast}}(\Omega)}=1 \big\}.
\end{equation}
Obviously, by formally choosing $\gamma =0$, the above minimization problem reduces to finding the best Sobolev constant $S_n$ in $\Omega$, which is known to be never achieved; as a consequence, we have the following upper bound for $S(\gamma)$:
$$S(\gamma)\leq S_n\quad\text{for every $0<\gamma< C_{emb}$}.$$

We also notice that, since we are assuming $0<\gamma<C_{emb}$ (recall that $C_{emb}$ is the constant defined in \eqref{eq:def_Cemb}), for every 
$u\in\mathcal{X}^{1,2}(\Omega)$ we have
\begin{align*}
 Q_\gamma(u) & = \int_{\Omega}|\nabla u|^{2} \, dx - \gamma \, [u]^{2}_{s}\geq 
 \left(1-\frac{\gamma}{C_{emb}}\right)\int_{\Omega}|\nabla u|^{2} \, dx \\
 & \geq \left(1-\frac{\gamma}{C_{emb}}\right)S_n > 0.
 \end{align*}
Gathering the last two formulas in display, we conclude that
\begin{equation} \label{eq:BoundSgammaUpLow}
	0<\left(1-\frac{\gamma}{C_{emb}}\right)S_n\leq S(\gamma)\leq S_n\quad\text{for every $0<\gamma< C_{emb}$}.
\end{equation}

Now, the key tool for the proof of Theorem~\ref{thm:Main} is the following.
\begin{proposition}\label{prop:Minimum_Achieved}
	Let $n \geq 3$. If $S(\gamma) < S_n$, the infimum in \eqref{eq:Def_Sgamma} is achieved.
	In this case, if $u_0\in \mathcal{X}^{1,2}(\Omega)$ achieves this infimum, then
	\begin{equation} \label{eq:tildeuSolutiondaMinimo}
\widetilde{u} = S(\gamma)^{1/(2^*-2)}u_0		
	\end{equation}
	is a nontrivial weak solution to problem \eqref{eq:Problem}.
	\end{proposition}
	
	\begin{proof}
		The argument and the notation that we will adopt is heavily influenced by the original one used by Brezis and Nirenberg \cite{BrNir}. In particular, since we will consider a minimizing sequence denoted by~$u_j$, we also denote by~$o(1)$ quantities that tend to zero as~$j\to+\infty$.
	To ease the readability, we split the proof into two steps.
	\medskip
	
	\textsc{Step I):} We first prove that, if $S(\gamma) < S_n$, then the infimum in \eqref{eq:Def_Sgamma} is achieved.
For this, let~$u_j$ be a minimizing sequence of functions in $\mathcal{X}^{1,2}(\Omega)$. In particular, it holds that $\|u_{j}\|_{L^{2^{\ast}}(\Omega)}=1$ for every $j\in \mathbb{N}$. 
In this way, we have that
\begin{equation}\label{vbncxm38945y43tugejk}
\lim_{j\to+\infty}Q_{\gamma}(u_j)=
S(\gamma).\end{equation}

Moreover, since $\gamma \in (0,C_{emb})$, it follows that 
		\begin{equation}\label{eq:Stima_Limitatezza}
			\left(1- \dfrac{\gamma}{C_{emb}}\right) \int_{\Omega}|\nabla u_j|^2 \, dx \leq Q_{\gamma}(u_j) .
		\end{equation}
From this and~\eqref{vbncxm38945y43tugejk},
we deduce that~$u_j$ is a bounded sequence in the Hilbert space~$\mathcal{X}^{1,2}(\Omega)$. Therefore, up to subsequences and using Remark~\ref{rem:PropX12Usare}-ii), we infer the existence of $u \in \mathcal{X}^{1,2}(\Omega)$ such that
		\begin{itemize}
			\item $u_j \to u$ weakly in $\mathcal{X}^{1,2}(\Omega)$;
			\item $[u_j - u]_s \to 0$;
			\item $u_j \to u$ a.e. in $\Omega$,
		\end{itemize}
		 as $j \to +\infty$. We stress that the latter implies that $\|u\|_{L^{2^{\ast}}(\Omega)}\leq 1$.
		
		Now, since $\|u_{j}\|_{L^{2^{\ast}}(\Omega)}=1$ for every $j\in \mathbb{N}$, it follows that
		\begin{equation*}
			\int_{\Omega}|\nabla u_j|^{2}\, dx \geq S_n,\quad \textrm{ for every } j \in \mathbb{N},
		\end{equation*}
		 and hence, recalling \eqref{vbncxm38945y43tugejk}, we get
		\begin{equation*}
			\begin{aligned}
			&\gamma \, [u]^{2}_{s}  = \gamma \, \lim_{j \to +\infty}[u_j]^{2}_{s} = \limsup_{j\to +\infty} \left(\int_{\Omega}|\nabla u_j|^{2} \, dx - Q_{\gamma}(u_j)\right)\\
			&\qquad
			\geq  \limsup_{j\to +\infty} \left(S_n - Q_{\gamma}(u_j)\right)	
			=S_n - S(\gamma) >0. 
		\end{aligned}
		\end{equation*}	
		This shows that $u \not\equiv 0$.
		
		Set now $v_j := u_j -u$ and notice that, as~$j\to+\infty$, 
		\begin{equation}\label{eq:Conto1}
			\begin{split}&
				\|u\|^2_{H^{1}_{0}(\Omega)} + \|v_j\|^2_{H^{1}_{0}(\Omega)} - \gamma [u]^2_{s}\\
				&\quad= \|u\|^2_{H^{1}_{0}(\Omega)} + \|u_j\|^2_{H^{1}_{0}(\Omega)} + \|u\|^2_{H^{1}_{0}(\Omega)} 
				- 2 \int_{\Omega}\nabla u \cdot \nabla u_j\, dx - \gamma [u]^2_{s}\\
				&\quad= 2\|u\|^2_{H^{1}_{0}(\Omega)} + Q_{\gamma}(u_j) + \gamma \big([u_j]^2_s - [u]^2_{s}\big) - 2\left(\|u\|^2_{H^{1}_{0}(\Omega)} + o(1)\right) \\
				&\quad= S(\gamma) + o(1) .
			\end{split}
		\end{equation}
		Moreover, since $v_j \to 0$ weakly in $\mathcal{X}^{1,2}(\Omega)$ (hence, $v_j$ is bounded in $\mathcal{X}^{1,2}(\Omega)$), it follows that $v_j$ is bounded in $L^{2^{\ast}}(\Omega)$ as well. This allows to use the Brezis-Lieb Lemma in~\cite{BrezisLieb} to get that,
		as~$j\to+\infty$,
		\begin{equation*}
				1 = \|u_j\|^{2^{\ast}}_{L^{2^{\ast}}(\Omega)} = \|u+ v_j\|^{2^{\ast}}_{L^{2^{\ast}}(\Omega)} 
= \|u\|^{2^{\ast}}_{L^{2^{\ast}}(\Omega)} + \|v_j\|^{2^{\ast}}_{L^{2^{\ast}}(\Omega)} + o(1) .
		\end{equation*}
		This yields, as~$j\to+\infty$,
		\begin{equation*}\label{eq:Stima_per_usare_BN}
		\begin{split}
			 1 & = \Big(\|u\|^{2^{\ast}}_{L^{2^{\ast}}(\Omega)} + \|v_j\|^{2^{\ast}}_{L^{2^{\ast}}(\Omega)} + o(1)\Big)^{2/2^*} \\
			 & \leq \|u\|^{2}_{L^{2^{\ast}}(\Omega)} + \|v_j\|^{2}_{L^{2^{\ast}}(\Omega)} + o(1)  
			 \\
			& \leq \|u\|^{2}_{L^{2^{\ast}}(\Omega)} + \dfrac{1}{S_n} \, \int_{\Omega}|\nabla v_j|^{2} \, dx + o(1) .
			 \end{split}
		\end{equation*}
		
		Recalling that $S(\gamma)>0$, we find that, 
		\begin{equation*}
			S(\gamma) \leq S(\gamma)\,  \|u\|^{2}_{L^{2^{\ast}}(\Omega)} + \dfrac{S(\gamma)}{S_n} \, \int_{\Omega}|\nabla v_j|^{2} \, dx + o(1) \quad \textrm { as } j \to +\infty,
		\end{equation*}
		 and thus, by using \eqref{eq:Conto1}, 
		\begin{equation*}
				\|u\|^2_{H^{1}_{0}(\Omega)} + \|v_j\|^2_{H^{1}_{0}(\Omega)} - \gamma [u]^{2}_{s} 
				\leq S(\gamma)\,  \|u\|^{2}_{L^{2^{\ast}}(\Omega)} + \dfrac{S(\gamma)}{S_n} \, \int_{\Omega}|\nabla v_j|^{2} \, dx + o(1) .
		\end{equation*}	
		From this, and exploiting that $S(\gamma)<S_n$, we get
		that, as~$j\to+\infty$,
		\begin{equation}\label{eq:Conto4}
				\|u\|^2_{H^{1}_{0}(\Omega)} - \gamma [u]^{2}_{s} \leq S(\gamma)  \|u\|^{2}_{L^{2^{\ast}}(\Omega)} + o(1) .
		\end{equation}

Now, since $u \not\equiv 0$, we can define
$$ u_0 := \frac{u}{\|u\|_{L^{2^*}(\Omega)}}$$
and, plugging it into ~\eqref{eq:Conto4} and taking the limit as $j\to +\infty$, we obtain 
$$ S(\gamma)\le Q_{\gamma}(u_0)\le S(\gamma).$$
This says, that~$u_0$ achieves the infimum in \eqref{eq:Def_Sgamma}.
		\medskip
		
		\textsc{Step II):} We now turn to show that, if $u_0$ achieves the infimum in \eqref{eq:Def_Sgamma}, then there exists a nontrivial
		solution of
		problem \eqref{eq:Problem}, given by \eqref{eq:tildeuSolutiondaMinimo}.
		
		To this end, we first observe that, if $u_0$ achieves the infimum in \eqref{eq:Def_Sgamma},
		by
		applying the Lagrange Multiplier Theorem we can infer the existence of $\mu \in \mathbb{R}$ such that,
for every~$\psi \in \mathcal{X}^{1,2}(\Omega)$,
		\begin{equation}\label{eq:Using_Lagrange}
			\begin{aligned}
 				 \int_{\Omega}\nabla u_0 \cdot \nabla \psi \, dx &- \gamma \iint_{\mathbb{R}^{2n}}\dfrac{(u_0(x)-u_0(y))(\psi(x)-\psi(y))}{|x-y|^{n+2s}}\, dx\, dy \\
 				&= \mu \int_{\Omega}|u_{0}|^{2^{\ast}-2}u_0\psi \, dx.
			\end{aligned}	
		\end{equation}
		In particular, we can choose $\psi:=u_0$ in \eqref{eq:Using_Lagrange} and this yields
		\begin{equation*}
			Q_{\gamma}(u_0) = \mu \int_{\Omega}|u_{0}|^{2^{\ast}}\, dx = \mu,
		\end{equation*}
		 which, recalling that $Q_{\gamma}(u_0) = S(\gamma)$, shows that
		$$\mu = S(\gamma).$$ 
		
		Setting now $\widetilde{u}:= k u_0$, with 
		$k:= S(\gamma)^{1/(2^{\ast}-2)}$, we get
		\begin{align*}
			\int_{\Omega}\nabla \widetilde{u} \cdot \nabla \psi \, dx &- \gamma \iint_{\mathbb{R}^{2n}}\dfrac{(\widetilde{u}(x)-\widetilde{u}(y))(\psi(x)-\psi(y))}{|x-y|^{n+2s}}\, dx\, dy \\
			& = S(\gamma)^{1/(2^*-2)+1}\int_{\Omega}|u_{0}|^{2^{\ast}-2}u_0\psi \, dx \\
			& = \int_{\Omega}|\widetilde{u}|^{2^{\ast}-2}\widetilde{u}\psi \, dx, 
		\end{align*}
for every $\psi\in \mathcal{X}^{1,2}(\Omega)$,
		 and this shows that $\widetilde{u}$ solves \eqref{eq:Problem}.
\end{proof}

In view of Proposition~\ref{prop:Minimum_Achieved}, the path towards establishing the existence of a weak solution to problem \eqref{eq:Problem} is laid out: indeed, it remains to show that 
\begin{equation} \label{eq:SgammaleqSntoprove} 
S(\gamma) < S_n.
\end{equation}
To prove this, we need to distinguish between two cases,
the high-dimensional case~$n\ge5$ and the low-dimensional case~$n=3,4$, which we now deal with separately.
\medskip

\noindent \textbf{1)\,\,The high-dimensional case $n\geq 5$.} In this first case,
we are able to prove that \eqref{eq:SgammaleqSntoprove} actually holds
\emph{for every $\gamma\in(0, C_{emb})$} by using an \emph{ad-hoc} competitor function
in \eqref{eq:Def_Sgamma}. To begin with, we have the following technical results.

\begin{lemma}\label{lem:andamenti}
  	Assume that $0 \in \Omega$. Let $\delta >0$ be
  	such that $B_{4\delta}\Subset \Omega$. Let~$\varphi \in C^{\infty}_{0}(\mathbb{R}^n,[0,1])$ be such that~$\varphi \equiv 1$ in~$B_{\delta}$ and~$\varphi \equiv 0$ in~$B^c_{2\delta}$.
  	
  	For every $\varepsilon>0$, let
  	\begin{equation}\label{eq:def_Ueps}
  		U_{\varepsilon}(x):= \varphi(x) \, V_{\varepsilon}(x), \quad {\mbox{for all }} x \in \mathbb{R}^{n},
  	\end{equation}
  	 where $V_{\varepsilon}$ is as in \eqref{eq:def_Talentiane} with $y=0$.
  	
  	Then, the following holds:
  	\begin{itemize}
  		\item[(i)] $\dfrac{\|\nabla V_{\varepsilon}\|^2_{L^2(\mathbb{R}^n)}}{\|V_{\varepsilon}\|^2_{L^{2^{\ast}}(\mathbb{R}^n)}} = S_n$;
  		\item[(ii)] $\|\nabla U_{\varepsilon}\|^2_{L^{2}(\Omega)} = K_1 + O(\varepsilon^{n-2}), \textit{as } \varepsilon \to 0^+$;
  		\item[(iii)] $\|U_{\varepsilon}\|^{2^{\ast}}_{L^{2^{\ast}}(\Omega)} = K_2 + O(\varepsilon^{n}), \textit{as } \varepsilon \to 0^+$,
  		where $\dfrac{K_1}{K_2^{1-2/n}}=S_n;$
  		\item[(iv)] $[U_{\varepsilon}]_{s}^{2} = O(\varepsilon^{2-2s}), \textit{as } \varepsilon \to 0^+$.
  	\end{itemize}   \end{lemma}
  	
  	\begin{proof}
  		The validity of (i) follows from the classical result by Talenti~\cite{Talenti2}. For (ii) and (iii) we refer to~\cite[Lemma~1.1]{BrNir}, while (iv) is contained in the proof of~\cite[Lemma~4.10]{BDVV5}.
  	\end{proof}
  
    We stress that the assumption $n\geq 5$ is actually not needed in the previous Lemma~\ref{lem:andamenti}: indeed, the typical term which exhibits different behaviours according to the dimension $n$ is $\|U_{\varepsilon}\|_{L^2(\R^n)}$. The restriction on $n$ becomes important in the upcoming estimate.
  
  \begin{lemma}\label{lem:crucial}
  	Let $n \geq 5$, $C_0>0$ and $v \in C^{1}(\mathbb{R}^{n})$ be such that
  	\begin{equation*}
  		|v(x)| + |x| |\nabla v(x)| \leq C_0 \, \min \left\{ 1, \dfrac{1}{|x|^{n-2}}\right\},  \quad \textrm{ for every } x \in \mathbb{R}^{n}\setminus \{0\}.
  	\end{equation*}
  	
  	Then, there exists~$C=C(n,C_0,s)>0$ such that 
  	\begin{equation}\label{eq:stima_cruciale}
  		\iint_{\mathbb{R}^{n} \times B^{c}_{R}}\dfrac{|v(x)-v(y)|^2}{|x-y|^{n+2s}}\, dx\,dy \leq \dfrac{C}{R^{2s}} \quad \textrm{ for all } R \geq 1.
  	\end{equation}
  	\begin{proof}
  		 Let~$\rho\in\left(0,\frac{R}2\right]$. We remark that,
if~$y\in B_R^c$ and~$z\in B_\rho(y)$,
  	$$	 |z|\ge |y|-|z-y|>|y|-\rho\ge|y|- \frac{R}2>\frac{|y|}2,$$
  		and therefore, if~$x\in B_\rho(y)$, we have that
  		\begin{eqnarray*}&&
  			|v(x)-v(y)|\le\sup_{z\in B_\rho(y)}|\nabla v(z)|\,|x-y|\le
  			\sup_{z\in B_\rho(y)}\frac{C_0\,|x-y|}{|z|^{n-1}}\le
  		\frac{C\,|x-y|}{|y|^{n-1}},
  		\end{eqnarray*}where we freely rename~$C>0$ from line to line.
 	
  		Consequently,
  		\begin{equation*}\begin{split}
  				&\!\!\!\!\!\!\!\!\iint_{\R^n\times B_R^c}\frac{|v(x)-v(y)|^2}{|x-y|^{n+2s}}\,dx\,dy\\
  				&\le\iint_{{\R^n\times B_R^c}\atop{\{|x-y|\le\rho\}}}\frac{|v(x)-v(y)|^2}{|x-y|^{n+2s}}\,dx\,dy+
  				2\iint_{{\R^n\times B_R^c}\atop{\{|x-y|>\rho\}}}\frac{|v(x)|^2+|v(y)|^2}{|x-y|^{n+2s}}\,dx\,dy
  				\\&\le
  				C\iint_{{\R^n\times B_R^c}\atop{\{|x-y|\le\rho\}}}\frac{|x-y|^{2-n-2s}}{|y|^{2n-2}}\,dx\,dy\\
  				& \qquad +
  				C\iint_{{\R^n\times B_R^c}\atop{\{|x-y|>\rho\}}}
  				\min\left\{1,\frac{1}{|x|^{2n-4}}\right\}
  				\frac{dx\,dy}{|x-y|^{n+2s}}\\&\qquad+C\iint_{{\R^n\times B_R^c}\atop{\{|x-y|>\rho\}}}
  				\frac{dx\,dy}{|y|^{2n-4}|x-y|^{n+2s}}
  				\\&\le \frac{C\rho^{2-2s}}{R^{n-2}}+\frac{C}{\rho^{2s}}\int_{\R^n}
  				\min\left\{1,\frac{1}{|x|^{2n-4}}\right\}\,dx+\frac{C}{R^{n-4}\rho^{2s}}\\&\le
  				\frac{C}{\rho^{2s}}\left( 
  				\frac{\rho^{2}}{R^{n-2}}+1
  				+\frac{1}{R^{n-4}}
  				\right).
  		\end{split}\end{equation*}
  		
  		In particular, choosing~$\rho:=\frac{R}2$, the left-hand side of~\eqref{eq:stima_cruciale} is controlled by
  		$$\frac{C}{R^{2s}}\left( 
  		\frac{1}{R^{n-4}}+1
  		\right)$$
  		and the desired result follows.
  	\end{proof}
  \end{lemma}
  
  \begin{proposition}\label{prop:Stima_Migliore}
  	Let $n\geq 5$ and let $U_{\varepsilon}$ be as in \eqref{eq:def_Ueps}. Then,
  	\begin{equation*}
  		[U_\varepsilon]^2_s = \varepsilon^{2-2s}\, [V_1]^{2}_{s} + O(\varepsilon^{2}),\quad \textrm{ as } \varepsilon \to 0^{+}.
  	\end{equation*}
  	\begin{proof} The proof follows the lines of
that of~\cite[Lemma~4.10]{BDVV5}.
Recall that~$\delta >0$ is such that~$B_{4\delta} \subset \Omega$ and define the sets 
  		\begin{align*}
  			& \mathbb{D}:= \left\{(x,y)\in \mathbb{R}^{2n}: x \in B_{\delta}, y\in B^c_{\delta} \, 
  			\textrm{ and } |x-y|> \delta/2 \right\}, \\
  			& \mathbb{E}:= \left\{(x,y)\in \mathbb{R}^{2n}: x \in B_{\delta}, y\in B^c_{\delta} \,
  			\textrm{ and } |x-y|\leq \delta/2 \right\}.
  		\end{align*}
  		With this notation, we get that
  		\begin{equation*}
  			\begin{aligned}
  				\iint_{\mathbb{R}^{2n}}\dfrac{|U_{\varepsilon}(x)-U_{\varepsilon}(y)|^2}{|x-y|^{n+2s}}\, dx\, dy &= \iint_{B_{\delta}\times B_{\delta}}\dfrac{|V_{\varepsilon}(x)-V_{\varepsilon}(y)|^2}{|x-y|^{n+2s}}\, dx\, dy \\
  				&+ 2  \iint_{\mathbb{D}}\dfrac{|U_{\varepsilon}(x)-U_{\varepsilon}(y)|^2}{|x-y|^{n+2s}}\, dx\, dy\\
  				&+ 2  \iint_{\mathbb{E}}\dfrac{|U_{\varepsilon}(x)-U_{\varepsilon}(y)|^2}{|x-y|^{n+2s}}\, dx\, dy\\
  				&+  \iint_{B^c_{\delta}\times B^c_{\delta}}\dfrac{|U_{\varepsilon}(x)-U_{\varepsilon}(y)|^2}{|x-y|^{n+2s}}\, dx\, dy.
  			\end{aligned}
  		\end{equation*}
  		Following the computations of~\cite[Proposition~21]{SeVa}, one can infer that 
  		the last three integrals of the above identity behave like $O(\varepsilon^{n-2})$ as $\varepsilon \to 0^{+}$. See also~\cite{daSilva} for similar computations with $p\neq 2$.
  		
 Let us now improve the estimate performed in~\cite{BDVV5} for the first integral (recall Lemma~\ref{lem:andamenti}-(iv)).
We perform the change of variables~$x:= \varepsilon \xi$ and~$y:=\varepsilon \eta$ to see that
\begin{equation}\label{bvncmxs328574839ty}
\iint_{B_{\delta}\times B_{\delta}}\dfrac{|V_{\varepsilon}(x)-V_{\varepsilon}(y)|^2}{|x-y|^{n+2s}}\, dx\, dy=
\iint_{B_{\delta/\varepsilon}\times B_{\delta/\varepsilon}}\dfrac{|V_{\varepsilon}(\varepsilon \xi)-V_{\varepsilon}( \varepsilon \eta)|^2}{| \varepsilon \xi- \varepsilon \eta|^{n+2s}}\,\varepsilon^{2n}\, d\xi\, d\eta	.
\end{equation}
We also observe that, for all~$z\in\R^n$,
$$ V_{\varepsilon}(\varepsilon z)
=\dfrac{\varepsilon^{(n-2)/2}}{(\varepsilon^{2} + |\varepsilon z|^2)^{(n-2)/2}}
=\dfrac{\varepsilon^{-(n-2)/2}}{(1+ |z|^2)^{(n-2)/2}} 
=\varepsilon^{-(n-2)/2}V_1(z).
$$
Plugging this information into~\eqref{bvncmxs328574839ty}, we infer that
\begin{equation}\label{vbcnx35y7777ghghkwhgu5i685i}
\iint_{B_{\delta}\times B_{\delta}}\dfrac{|V_{\varepsilon}(x)-V_{\varepsilon}(y)|^2}{|x-y|^{n+2s}}\, dx\, dy=\varepsilon^{2-2s}
\iint_{B_{\delta/\varepsilon}\times B_{\delta/\varepsilon}}\dfrac{|V_{1}(\xi)-V_{1}( \eta)|^2}{| \xi- \eta|^{n+2s}}\, d\xi\, d\eta	.
\end{equation}

Moreover, we notice that
\begin{eqnarray*}
\iint_{\mathbb{R}^{2n} \setminus \left(B_{\delta/\varepsilon} \times B_{\delta/\varepsilon}\right)} \dfrac{|V_1(\xi)-V_1(\eta)|^2}{|\xi -\eta|^{n+2s}}\, d\xi \,d\eta
\leq 2\iint_{\mathbb{R}^{n} \times B^c_{\delta/\varepsilon}} \dfrac{|V_1(\xi)-V_1(\eta)|^2}{|\xi -\eta|^{n+2s}}\, d\xi \,d\eta.
\end{eqnarray*}
Now, since~$V_1$ satisfies the assumptions of Lemma~\ref{lem:crucial}, choosing $R := \delta/\varepsilon$, it follows that there exists~$\widetilde{C}>0$ (depending on $\delta$ but independent of $\varepsilon$) such that
  		\begin{equation*}
  \iint_{\mathbb{R}^{2n} \setminus \left(B_{\delta/\varepsilon} \times B_{\delta/\varepsilon}\right)} \dfrac{|V_1(\xi)-V_1(\eta)|^2}{|\xi -\eta|^{n+2s}}\, d\xi \,d\eta\leq \widetilde{C} \varepsilon^{2s}.
  		\end{equation*}
  		
{F}rom this and~\eqref{vbcnx35y7777ghghkwhgu5i685i}, we conclude that	
\begin{eqnarray*}
&&\iint_{B_{\delta}\times B_{\delta}}\dfrac{|V_{\varepsilon}(x)-V_{\varepsilon}(y)|^2}{|x-y|^{n+2s}}\, dx\, dy\\&=&\varepsilon^{2-2s}
\left(\iint_{\mathbb{R}^{2n}} \dfrac{|V_1(\xi)-V_1(\eta)|^2}{|\xi -\eta|^{n+2s}}\, d\xi \,d\eta
-\iint_{\mathbb{R}^{2n} \setminus \left(B_{\delta/\varepsilon} \times B_{\delta/\varepsilon}\right)} \dfrac{|V_1(\xi)-V_1(\eta)|^2}{|\xi -\eta|^{n+2s}}\, d\xi \,d\eta\right)
\\&=&\varepsilon^{2-2s}[V_1]_s^2+O(\varepsilon^2).
\end{eqnarray*}
Combining all the estimates, we get the desired result.
\end{proof}
  \end{proposition}
  
Thanks to the previous results, we can now prove that the family $U_\varepsilon$ decreases the energy of $Q_{\gamma}(\cdot)$ below the threshold $S_n$, that is,
\eqref{eq:SgammaleqSntoprove} holds.
\begin{proposition}\label{prop:Using_Ueps}
	Let $n\geq 5$, and let $\gamma\in (0, C_{emb})$. Then, there exists a function~$v \in \mathcal{X}^{1,2}(\Omega)$ such that $Q_{\gamma}(v)<S_{n}$.
	\begin{proof}
		Let $U_{\varepsilon}$ be as in \eqref{eq:def_Ueps}. By Lemma~\ref{lem:andamenti} and Proposition~\ref{prop:Stima_Migliore}, it follows that, as~$\varepsilon \to 0^{+}$,
		\begin{eqnarray*}
			&&Q_{\gamma}\left(\frac{U_{\varepsilon}}{\|U_{\varepsilon}\|_{L^{2^{\ast}}(\Omega)}}\right) =\frac1{
			\|U_{\varepsilon}\|^{2}_{L^{2^{\ast}}(\Omega)}}
			\left(\int_{\Omega}|\nabla U_{\varepsilon}|^{2} \, dx 
			- \gamma \, [U_{\varepsilon}]^{2}_{s} \right)\\&&\qquad
			=\frac1{K_2^{{2}/{2^{\ast}} } +O(\varepsilon^{n})}\left(
			K_1+O(\varepsilon^{n-2})
			-\gamma [V_1]^{2}_{s} \varepsilon^{2-2s} +O(\varepsilon^2)
			\right)\\&&\qquad
			=\left(\frac1{K_2^{{2}/{2^{\ast}} }} +O(\varepsilon^{n})\right)\left(
			K_1+O(\varepsilon^{n-2})
			-\gamma [V_1]^{2}_{s} \varepsilon^{2-2s} +O(\varepsilon^2)
			\right)	
			\\&&\qquad=
			S_{n} - \frac{\gamma [V_1]^{2}_{s}}{ K_2^{{2}/{2^{\ast}}}} \,\varepsilon^{2-2s} + o(\varepsilon^{2-2s}).
		\end{eqnarray*}
	This entails the desired result.	
	\end{proof}
\end{proposition}

With Proposition~\ref{prop:Using_Ueps} at hand, we can immediately prove the first assertion in our main Theorem~\ref{thm:Main}.
\begin{proof}[Proof (of Theorem~\ref{thm:Main} - (1))]
	Assume that $n\geq 5$
	and let $\gamma\in (0,C_{emb})$. Owing to Proposition
	\ref{prop:Using_Ueps}, we know that $S(\gamma) < S_n$; hence, by Proposition
	\ref{prop:Minimum_Achieved} we conclude that $S(\gamma)$ is achieved, and 
	there exists a weak solution of
	problem \eqref{eq:Problem}.
\end{proof}

\noindent\textbf{2)\,\,The low-dimensional case $n = 3,4$.} In this second case, even if we can still rely on Proposition~\ref{prop:Minimum_Achieved} (which holds for every $n\geq 3$), we \emph{cannot use} the estimate in Proposition~\ref{prop:Stima_Migliore}; rather, we make use of a continuity argument to show that
$$S(\gamma)<S_n\quad\text{for every $\gamma^*<\gamma < C_{emb}$},$$
for some (not explicit) $\gamma^*\in [0,C_{emb})$.
\medskip

To this end, we consider $S(\cdot)$ as a \emph{function of $\gamma$}, defined
in the interval~$\mathcal{J} := (0,C_{emb}$). Owing to the definition of $S(\gamma)$ in \eqref{eq:Def_Sgamma}, we see that
\begin{equation} \label{eq:SgammaMonotona} 
S(\lambda)\leq S(\mu)\quad\text{for every $0<\mu\leq\lambda< C_{emb}$},	
\end{equation}
that is, $S(\cdot)$ is non-increasing on $\mathcal{J}$.

Moreover, we have the following:
\begin{lemma} \label{lem:ContinuityS}
The function $S(\cdot)$ is continuous on $\mathcal{J}$. Furthermore,
\begin{equation} \label{eq:LimitSgammaBoundary} 
\lim_{\gamma\to 0^+}S(\gamma) = S_n \qquad{\mbox{and}}\qquad \lim_{\gamma\to C_{emb}} S(\gamma) = 0.	
\end{equation}
 \end{lemma}
 \begin{proof}
 To ease the readability, we split the proof into two steps.
 \medskip
 
 \textsc{Step I):} We begin by proving the continuity of $S(\cdot)$ in $\mathcal{J}$. To this end we first observe that, since $S(\cdot)$ is non-increasing, for every $\gamma\in \mathcal{J}$ we have
  $$\exists\,\,\,{S}(\gamma-) := \lim_{\mu\to \gamma^-}{S}(\mu)\in\R\quad
  \text{and}\quad \exists\,\,\,{S}(\gamma+) := \lim_{\mu\to \gamma^+}{S}(\mu)\in\R.$$ 

We now fix $\gamma_0 \in\mathcal{J}$ and $r > 0$ such that $[\gamma_0-r,\gamma_0+r]\subseteq\mathcal{J}$ and we show that~${S}(\cdot)$ is 
  continuous from both the left and the right at~$\gamma_0$, that is,
  $${S}(\gamma_0-) = {S}(\gamma_0+) = {S}(\gamma_0).$$
  
  To prove the left-continuity, we proceed as follows. First of all,
  given any $\e\in(0,1)$, 
  by the definition of ${S}(\gamma_0)$ we infer that there exists 
  $u = u_{\e,\gamma_0}\in \mathcal{X}^{1,2}(\Omega)$ such that
  \begin{equation} \label{eq:choiceuquasiminimo}
\|u\|_{L^{2^*}(\Omega)} = 1\qquad\text{and}\qquad 
   {S}(\gamma_0)\leq {Q}_{\gamma_0}(u) < {S}(\gamma_0)+\frac{\e}{2}.  	
  \end{equation}
  From this, using the monotonicity of ${S}(\cdot)$,  for every $\gamma \leq \gamma_0$ we obtain
  \begin{equation}\label{bvcnmxoi3r48y328tyrfjkq098765}\begin{split}
   & 0 \leq {S}(\gamma)-{S}(\gamma_0)
   \leq {Q}_{\gamma}(u)-{S}(\gamma_0)\\
   & \qquad = \big({Q}_{\gamma_0}(u)-{S}(\gamma_0)\big)
   + (\gamma_0-\gamma)[u]^2_s < \frac{\e}{2}+(\gamma_0-\gamma)[u]^2_s.
  \end{split}\end{equation}
  
Moreover, exploiting \eqref{eq:def_Cemb}, 
$$ Q_{\gamma_0}(u)=\|\nabla u\|^2_{L^2(\Omega)}-\gamma_0[u]^2_s\ge 
\left(1-\frac{\gamma_0}{C_{emb}}\right)\|\nabla u\|^2_{L^2(\Omega)}.$$
Therefore, using again~\eqref{eq:def_Cemb} and recalling~\eqref{eq:choiceuquasiminimo},
\begin{equation*}\begin{split}
&[u]^2_s \leq \frac{1}{C_{emb}}\|\nabla u\|^2_{L^2(\Omega)} 
\le\frac{1}{C_{emb}}\left(1-\frac{\gamma_0}{C_{emb}}\right)^{-1}
Q_{\gamma_0}(u)\\
&\qquad<\frac{1}{C_{emb}-\gamma_0}\left(S(\gamma_0)+\frac{\e}{2}\right)\le \frac{S(\gamma_0)+1}{C_{emb}-\gamma_0}
.\end{split}
\end{equation*}
From this and \eqref{eq:BoundSgammaUpLow} we thus obtain that
\begin{equation} \label{eq:StimaUnifGagliardo}
 [u]^2_s <  \frac{S_n+1}{C_{emb}-\gamma_0} \le
\frac{S_n+1}{C_{emb}-\gamma_0-r} =:\Theta(\gamma_0).
\end{equation}

Using this estimate together with~\eqref{bvcnmxoi3r48y328tyrfjkq098765},
we find that
\begin{equation*} 
0 \leq {S}(\gamma)-{S}(\gamma_0)
< \frac{\e}{2}+(\gamma_0-\gamma)\Theta(\gamma_0).\end{equation*}
As a consequence, setting 
\begin{equation}\label{vbncmd322oyt3ugweh0987654kjhgfd}
\delta = \delta_{\e,\gamma_0} := \min\left\{r, \frac{\e}{2\Theta(\gamma_0)}\right\} >0,\end{equation} we conclude that,
for every $ \gamma\in(\gamma_0-\delta, \gamma_0]$,
  $$0 \leq {S}(\gamma)-{S}(\gamma_0) < 
  \e$$
and this proves that ${S}(\cdot)$  is continuous from the left at $\gamma_0$.

As regards the continuity from the right, we proceed {essentially as above},
  but we exploit in a more crucial way the uniform estimate \eqref{eq:StimaUnifGagliardo}.
  First of all, given~$\e > 0$ and~$\gamma\in[\gamma_0,\gamma_0+r]$, we let
  $u = u_{\e,\gamma}\in\mathcal{X}^{1,2}(\Omega)$ be such that
  $$\|u\|_{L^{2^*}(\Omega)} = 1\qquad\text{and}\qquad {S}(\gamma)\leq {Q}_{\gamma}(u) < {S}(\gamma)+\frac{\e}{2}.$$
  From this, by the monotonicity of ${S}(\cdot)$, we obtain 
  \begin{align*}
   & 0 \leq {S}(\gamma_0)-{S}(\gamma)
   \leq {Q}_{\gamma_0}(u)-{S}(\gamma) \\
   & \qquad = \big({Q}_{\gamma}(u)-{S}(\gamma)\big)
   + (\gamma-\gamma_0)[u]^2_s \quad < \frac{\e}{2}+(\gamma-\gamma_0)[u]^2_s.
  \end{align*}

Furthermore, by arguing exactly as above
to obtain the estimate in~\eqref{eq:StimaUnifGagliardo} (with~$\gamma_0$ here replaced by $\gamma$), we obtain that
  \begin{equation*}
[u]^2_s  < \frac{1}{C_{emb}-\gamma}\left(S(\gamma)+\frac{\e}{2}\right)
	 \leq \frac{S_n+1}{C_{emb}-\gamma}< \frac{S_n+1}{C_{emb}-\gamma_0-r} = \Theta(\gamma_0).
\end{equation*}
As a consequence, taking~$\delta$ as in~\eqref{vbncmd322oyt3ugweh0987654kjhgfd},
we conclude that, for every~$\gamma \in[\gamma_0,\gamma_0+\delta_\e)$,
$$0 \leq {S}(\gamma_0)-{S}(\gamma) < \e.$$
We then conclude that ${S}(\cdot)$ 
is also continuous from the right at $\gamma_0$, and thus, by the arbitrariness of $\gamma_0$, $S(\cdot)\in C(\mathcal{J})$.
  \medskip
  
  \textsc{Step II):} We now prove the limits in~\eqref{eq:LimitSgammaBoundary}. The limit of $S(\cdot)$ as~$\gamma\to 0^+$ is a direct consequence of
  \eqref{eq:BoundSgammaUpLow}, hence we now focus on the limit as $\gamma\to C_{emb}$. 
  
We recall Remark~\ref{rem:PropX12Usare}-iii) and we take~$\Phi_0\in\mathcal{X}^{1,2}(\Omega)$ such that
$$[\Phi_0]_s^2 = 1\qquad\text{and}\qquad \|\nabla\Phi_0\|^2_{L^2(\Omega)} = C_{emb}.$$ 
Then, we have that, for every $\gamma\in\mathcal{J}$,
\begin{align*}
0\leq S(\gamma) \leq Q_\gamma(\Phi_0) = (C_{emb}-\gamma)[\Phi_0]^2_s = C_{emb}-\gamma.
\end{align*}
Hence, by letting $\gamma\to C_{emb}$ in the above estimate, we obtain the desired limit. This ends the proof.
\end{proof}
  
Thanks to Lemma~\ref{lem:ContinuityS}, we can prove the second assertion in Theorem~\ref{thm:Main}.

\begin{proof}[Proof (of Theorem~\ref{thm:Main} - (2))]
	Assume that $n = 3,4$ and recall the notation~$\mathcal{J}=(0,C_{emb})$. 
	
By~\eqref{eq:BoundSgammaUpLow}, we know that~$S(\gamma) > 0$ for all~$\gamma\in\mathcal{J}$. Thus,
owing to~\eqref{eq:SgammaMonotona} and Lemma~\ref{lem:ContinuityS}, we have that
	\begin{equation}\label{OSM:CO}{\mbox{either $S(\mathcal{J}) = (0,S_n)$ or~$S(\mathcal{J}) = (0,S_n]$}}.\end{equation}
  Concerning this statement, we point that, in our setting, we are not able to
	exclude that $S_n\in S(\mathcal{J})$, as we cannot guarantee that $S(\gamma) < S_n$ on $\mathcal{J}$,
	and this is the reason for which the second option in~\eqref{OSM:CO} cannot be excluded.
	 
	In any case, it follows from~\eqref{OSM:CO} that
	$$\{\gamma\in\mathcal{J}:\,S(\gamma)<S_n\}\ne\varnothing.$$
	As a consequence, we can define
	$$\gamma^* := \inf\{\gamma\in\mathcal{J}:\,S(\gamma)<S_n\}.$$ 
	We also deduce from~\eqref{OSM:CO} that $\gamma^*\in[0,C_{emb})$.
	Moreover,
	\begin{itemize}
		\item[i)] $S(\gamma) = S_n$ for all $\gamma\in(0, \gamma^*)$;
		\item[ii)] $0<S(\gamma)<S_n$ for all $\gamma\in(\gamma^*,C_{emb})$.
	\end{itemize}
	Hence, for every fixed $\gamma\in(\gamma^*,C_{emb})$, we  can
	 apply Proposition
	\ref{prop:Minimum_Achieved}, ensuring that~$S(\gamma)$ is achieved and 
	that there exists a weak solution of
	problem \eqref{eq:Problem}.
	\end{proof}
 
\end{document}